\documentclass[12pt]{article}

\usepackage{amsmath, epsfig, cite}
\usepackage{amssymb}
\usepackage{amsfonts}
\usepackage{latexsym}
\usepackage{amsthm}

\newtheorem{thm}{Theorem}[section]

\newtheorem{cor}[thm]{Corollary}

\newtheorem{lem}[thm]{Lemma}

\numberwithin{equation}{section}

\renewcommand{\thefootnote}{}

\begin{document}

\begin{center}
{\large\bf Some multidimensional Rogers--Ramanujan type identities

 \footnote{Corresponding author$^*$. Email addresses: weichuanan78@163.com (C. Wei), yuanboyu81@163.com (Y. Yu), 378075837@qq.com (G. Ruan)}}
\end{center}

\renewcommand{\thefootnote}{$\dagger$}

\vskip 2mm \centerline{Chuanan Wei, Yuanbo Yu, Guozhu Ruan$^*$}
\begin{center}
{School of Intelligent Medicine and Technology\\ Hainan Medical
University, Haikou 571199, China}
\end{center}


\vskip 0.7cm \noindent{\bf Abstract.} With the help of the contour
integral method, we derive a parametric reduction formula that
transforms a double series into a single series. This formula
recovers two results of Uncu and Zudilin as well as two results of
Cao and Wang, and it is also connected with an identity due to
Berkovich and Warnaar. In addition, we obtain several triple-sum
generalizations of Cao and Wang's formulas. As applications, we
present a number of multidimensional Rogers--Ramanujan type
identities, both with and without parameters.

\vskip 3mm \noindent {\it Keywords}: Rogers--Ramanujan type
identities; Jacobi's product triple identity; contour integral
method; bisection method

 \vskip 0.2cm \noindent{\it AMS
Subject Classifications:} 33D15; 11A07; 11B65

\section{Introduction}

 For two complex numbers $x$, $q$ with $|q|<1$ and a nonnegative integer $n$, define the $q$-shifted factorial
 to be
 \begin{equation*}
(x;q)_{\infty}=\prod_{k=0}^{\infty}(1-xq^k)\quad\text{and}\quad
(x;q)_n=\frac{(x;q)_{\infty}}{(xq^n;q)_{\infty}}.
 \end{equation*}
For simplicity, we usually adopt the compact notation
\begin{equation*}
(x_1,x_2,\dots,x_m;q)_{n}=(x_1;q)_{n}(x_2;q)_{n}\cdots(x_m;q)_{n},
 \end{equation*}
where $m\in\mathbb{Z}^{+}$ and $n\in\mathbb{Z}^{+}\cup\{0,\infty\}.$
Then the famous Rogers--Ramanujan identities can be stated as
follows:
\begin{align}
&
\sum_{k=0}^{\infty}\frac{q^{k^2}}{(q;q)_k}=\frac{1}{(q,q^4;q^5)_{\infty}},
\label{RR-a}\\[1mm]
 &
\sum_{k=0}^{\infty}\frac{q^{k^2+k}}{(q;q)_k}=\frac{1}{(q^2,q^3;q^5)_{\infty}}.
\label{RR-b}
\end{align}
The bilateral generalizations of them can be seen in Schlosser
\cite{Schlosser}.

In 2019,  Kanade and Russell \cite{Kanade} proposed nine conjectured
multidimensional Rogers--Ramanujan type identities related to level
2 characters of the affine Lie algebra $A_9^{(2)}$. Five of them are
confirmed by Bringmann, Jennings-Shaffer, and Mahlburg \cite{BJM}.
Rosengren \cite{Rosengren} proved all of the nine formulas by the
contour integral method.

Recently, Uncu and Zudilin \cite{Uncu} gave the following two
interesting identities:
\begin{align}
&\sum_{j,\,k\geq0}\frac{q^{j^2+2jk+2k^2}}{(q;q)_j(q^2;q^2)_k}
=\frac{(q^3;q^3)_{\infty}^2}{(q;q)_{\infty}(q^6;q^6)_{\infty}},
\label{UZ-a}\\[1mm]
&\sum_{j,\,k\geq0}\frac{q^{j^2+2jk+2k^2+j+2k}}{(q;q)_j(q^2;q^2)_k}
=\frac{(q^6;q^6)_{\infty}^2}{(q^2;q^2)_{\infty}(q^3;q^3)_{\infty}}.
\label{UZ-b}
\end{align}
Ole Warnaar has pointed that \eqref{UZ-a} and \eqref{UZ-b} are
instances of Bressoud's results from \cite{Bressoud}. Though the
contour integral method, Wang \cite{Wang} recovered \eqref{UZ-a} and
\eqref{UZ-b} and Cao and Wang \cite[Theorem 3.8]{Cao} proved the
following two formulas:
\begin{align}
&\sum_{j,\,k\geq0}\frac{q^{j^2+2jk+2k^2}}{(q;q)_j(q^2;q^2)_k}(-1)^jx^{j+k}
=(qx;q^2)_{\infty},
 \label{wang-a}
\\[1mm]
&\sum_{j,\,k\geq0}\frac{q^{j^2+2jk+2k^2+k}}{(q;q)_j(q^2;q^2)_k}x^{j+2k}
=(-qx;q)_{\infty},
  \label{wang-b}
\end{align}
where $x$ is an arbitrary complex number. It should be mentioned
that Berkovich and Warnaar \cite[Equation (3.10)]{Berkovich}
contains the following simple transformation formula between two
single series:
\begin{align}\label{Berkovich}
 \sum_{k=0}^{\infty}\frac{q^{2k^2-k}y^{2k}}{(q^2;q^2)_k}
=(y;q)_{\infty}\sum_{k=0}^{\infty}\frac{q^{\binom{k}{2}}y^k}{(q;q)_k(y;q)_k}.
\end{align}

For more conclusions, the reader is referred to the papers
\cite{Cao-b,Hao, Kur,Ramanujan,Nancy,Slater,Wang-b}. Inspired by
these works, we shall establish the following parametric reduction
formula that transforms a double series into a single series.

\begin{thm}\label{thm-a}
Let $x$ and $y$ be complex numbers. Then
\begin{align}
\sum_{j,\,k\geq0}\frac{q^{j^2+2jk+2k^2-j-k}}{(q;q)_j(q^2;q^2)_k}x^{j}y^{2k}
=(y;q)_{\infty}
\sum_{k=0}^{\infty}\frac{(-x/y;q)_k}{(q;q)_k(y;q)_k}q^{\binom{k}{2}}y^k.
\label{eq:wei-a}
\end{align}
\end{thm}

Theorem \ref{thm-a} is a united generalization of
\eqref{UZ-a}-\eqref{Berkovich}. The corresponding details are
displayed as follows.

Choosing $(x,y)=(q,q^{\frac{1}{2}})$ in Theorem \ref{thm-a} and then
calculating the series on the right-hand side by Ramanujan's formula
(cf. \cite[Entry 5.3.2]{Andrews-b}):
\begin{align*}
 \sum_{k=0}^{\infty}\frac{(-q;q^2)_k}{(q;q)_{2k}}q^{k^2}
=\frac{(q^6;q^6)_{\infty}^2}{(q;q)_{\infty}(q^{12};q^{12})_{\infty}},
\end{align*}
we catch hold of the equation \eqref{UZ-a}.

Fixing $(x,y)=(q^2,q^{\frac{3}{2}})$ in Theorem \ref{thm-a} and then
evaluating the series on the right-hand side by Ramanujan's another
formula (cf. \cite[Entry 3.4.4]{Andrews-b}):
\begin{align*}
 \sum_{k=0}^{\infty}\frac{(-q;q^2)_k}{(q;q)_{1+2k}}q^{k^2+2k}
=\frac{(q^{12};q^{12})_{\infty}(-q^6;q^6)_{\infty}}{(q;q)_{\infty}(-q^{2};q^{2})_{\infty}},
\end{align*}
we grasp hold of the identity \eqref{UZ-b}.

 Taking $(x,y)\to (-qx,-q^{\frac{1}{2}}x^{\frac{1}{2}})$ in Theorem \ref{thm-a} and then computing the
series on the right-hand side by Euler's $q$-exponential formula
(cf. \cite[Appendix (II.2)]{Gasper}):
\begin{align}
 \sum_{k=0}^{\infty}\frac{q^{\binom{k}{2}}z^{k}}{(q;q)_k}=(-z;q)_{\infty},
\label{Euler-a}
\end{align}
 we arrive at the equation  \eqref{wang-a}.

  Letting $(x,y)\to (qx,-qx)$ in Theorem \ref{thm-a}, we are led to the identity \eqref{wang-b}.

  It is also obvious that Theorem \ref{thm-a} reduces to \eqref{Berkovich}
   when the parameter $x$ equals $0$.

Subsequently, we shall display the following triple-sum
generalization of \eqref{wang-a}.

\begin{thm}\label{thm-d}
Let $x$ and $y$ be complex numbers. Then
\begin{align}\label{eq:wei-d}
\sum_{j,\,k,\,\ell\,\geq0}\frac{(x;q)_j(-x)^{k+2\ell}y^{k+\ell}}{(q;q)_j(q;q)_k(q^2;q^2)_{\ell}}
q^{j+\binom{k}{2}+\binom{j+k+2\ell}{2}}
=(qx,xy;q^2)_{\infty}(-q;q)_{\infty}.
\end{align}
\end{thm}

When $x=1$, Theorem \ref{thm-d} reduces to \eqref{wang-a} thanks to
the relation (cf. \cite[P. 24]{Gasper}):
$(q,-q,-q^2;q^2)_{\infty}=1,$ which will be utilized without
indication elsewhere.

Furthermore, we shall give the following two triple-sum
generalizations of \eqref{wang-b}.

\begin{thm}\label{thm-e}
Let$x$ and $y$ be complex numbers. Then
\begin{align}
&\sum_{j,\,k,\,\ell\,\geq0}\frac{(x^2y^2;q^2)_kx^jy^{2j+2\ell}}{(q;q)_j(q^2;q^2)_k(q^2;q^2)_{\ell}}
(-1)^{j+k}q^{(j+k+\ell)(j+k+\ell-1)+\ell^2+k}
\notag\\[1mm]
&\quad=\:\frac{(q;q^2)_{\infty}}{2}\big\{(xy,-y;q)_{\infty}+(-xy,y;q)_{\infty}\big\}.
\label{eq:wei-e1}
\end{align}
\end{thm}

\begin{thm}\label{thm-f}
Let $x$ and $y$ be complex numbers. Then
\begin{align}
&\sum_{j,\,k,\,\ell\,\geq0}\frac{(x^2y^2;q^2)_kx^jy^{2j+2\ell}}{(q;q)_j(q^2;q^2)_k(q^2;q^2)_{\ell}}
(-1)^{j+k}q^{(j+k+\ell)(j+k+\ell-1)+\ell^2+3k}
\notag\\[1mm]
&\quad=\:\frac{(q;q^2)_{\infty}}{2(y/q-xy)}\big\{(xy,-y/q;q)_{\infty}-(-xy,y/q;q)_{\infty}\big\}.
\label{eq:wei-f}
\end{align}
\end{thm}

When $xy=1$, Theorems \ref{thm-e} and \ref{thm-f} both become
\eqref{wang-b}. From the two theorems, we can also deduce some new
multidimensional Rogers--Ramanujan type identities.

Taking $(x,y)\to(-xq^m,q^{-m})$ in Theorem \ref{thm-e}, there holds
the following formula.

\begin{cor}\label{cor-b}
Let $x$ be a complex number and let $m$ be a nonnegative integer.
Then
\begin{align}
&\sum_{j,\,k,\,\ell\,\geq0}\frac{(x^2;q^2)_k\,x^j(-1)^{k}}{(q;q)_j(q^2;q^2)_k(q^2;q^2)_{\ell}}
q^{(j+k+\ell)(j+k+\ell-1)+\ell^2+k-m(j+2\ell)}
\notag\\[1mm]
&\quad=(-q^{-m};q)_m(-x;q)_{\infty}.
 \label{eq:wei-g}
\end{align}
\end{cor}

Taking $(x,y)\to(1,x^{\frac{1}{2}})$, $(q,x^{\frac{1}{2}}/q)$ or
$(1/q,qx^{\frac{1}{2}})$ in Theorem \ref{thm-e},
 we find the following  three conclusions.

\begin{cor}\label{cor-c}
Let $x$ be a complex number. Then
\begin{align*}
&\sum_{j,\,k,\,\ell\,\geq0}\frac{(x;q^2)_k\,x^{j+\ell}(-1)^{j+k}}{(q;q)_j(q^2;q^2)_k(q^2;q^2)_{\ell}}
q^{(j+k+\ell)(j+k+\ell-1)+\ell^2+k}=\:(q,x;q^2)_{\infty},
\\[1mm]
&\sum_{j,\,k,\,\ell\,\geq0}\frac{(x;q^2)_k\,x^{j+\ell}(-1)^{j+k}}{(q;q)_j(q^2;q^2)_k(q^2;q^2)_{\ell}}
q^{(j+k+\ell)(j+k+\ell-1)+\ell^2-j+k-2\ell}=\:(q,x;q^2)_{\infty},
\\[1mm]
&\sum_{j,\,k,\,\ell\,\geq0}\frac{(x;q^2)_k\,x^{j+\ell}(-1)^{j+k}}{(q;q)_j(q^2;q^2)_k(q^2;q^2)_{\ell}}
q^{(j+k+\ell)(j+k+\ell-1)+\ell^2+j+k+2\ell}=\:(q,q^2x;q^2)_{\infty}.
\end{align*}
\end{cor}

Taking $(x,y)\to(-xq^{m-1},q^{1-m})$ in Theorem \ref{thm-f}, we get
the following result.

\begin{cor}\label{cor-d}
Let $x$ be a complex number and let $m$ be a nonnegative integer.
Then
\begin{align}
&\sum_{j,\,k,\,\ell\,\geq0}\frac{(x^2;q^2)_k\,x^j(-1)^{k}}{(q;q)_j(q^2;q^2)_k(q^2;q^2)_{\ell}}
q^{(j+k+\ell)(j+k+\ell-1)+\ell^2+3k-(m-1)(j+2\ell)}
\notag\\
&\quad=\frac{(-q^{-m};q)_m(-x;q)_{\infty}}{q^{-m}+x}.
 \label{eq:wei-h}
\end{align}
\end{cor}

Taking $(x,y)\to(1,x^{\frac{1}{2}})$ or $(1/q^2,q^2x^{\frac{1}{2}})$
in Theorem \ref{thm-f}, we discover the following two conclusions.

\begin{cor}\label{cor-e}
Let $x$ be a complex number. Then
\begin{align*}
&\sum_{j,\,k,\,\ell\,\geq0}\frac{(x;q^2)_k\,x^{j+\ell}(-1)^{j+k}}{(q;q)_j(q^2;q^2)_k(q^2;q^2)_{\ell}}
q^{(j+k+\ell)(j+k+\ell-1)+\ell^2+3k}=(q^3,x;q^2)_{\infty},
\\[1mm]
&\sum_{j,\,k,\,\ell\,\geq0}\frac{(x;q^2)_k\,x^{j+\ell}(-1)^{j+k}}{(q;q)_j(q^2;q^2)_k(q^2;q^2)_{\ell}}
q^{(j+k+\ell)(j+k+\ell-1)+\ell^2+2j+3k+4\ell}=(q^3,q^2x;q^2)_{\infty}.
\end{align*}
\end{cor}

The rest of the paper is arranged as follows. In terms of the
contour integral method, we shall prove Theorem \ref{thm-a} in
Section 2. Similarly, the proof of Theorems \ref{thm-d}-\ref{thm-f}
will be provided in Section 3. Via the bisection method, several new
multidimensional Rogers--Ramanujan type identities will be deduced
in Section 4.

\section{Proof of Theorems \ref{thm-a}}

Let $f(z)$ represent the coming series
$$f(z)=\sum_{k=-\infty}^{\infty}a(k)z^k.$$ We employ $[z^k]f(z)$
to stand for the coefficient of $z^k$. Then it is well-known that
\begin{align*}
\oint f(z)\frac{dz}{2\pi iz}=[z^0]f(z),
\end{align*}
where the contour is positively oriented and simple closed around
the origin.

For the aim to prove Theorem \ref{thm-a}, we need the following
lemma (cf. \cite[Equation (4.10.6)]{Gasper}).

\begin{lem}\label{lem-a}
Assume that
\begin{align*}
P(z)=\frac{(a_1z,\ldots,a_Az,b_1/z,\ldots,b_B/z;q)_{\infty}}{(c_1z,\ldots,c_Az,d_1/z,\ldots,d_D/z;q)_{\infty}}
\end{align*}
has only simple poles and $|a_1\cdots a_A/c_1\cdots c_A|<1$. Then
\begin{align*}
\oint P(z)\frac{dz}{2\pi iz}
&=\frac{(b_1c_1,\ldots,b_Bc_1,a_1/c_1,\ldots,a_A/c_1;q)_{\infty}}{(q,d_1c_1,\ldots,d_Dc_1,c_2/c_1,\ldots,c_A/c_1;q)_{\infty}}
\\[1mm]
&\quad\times\sum_{k=0}^{\infty}\frac{(d_1c_1,\ldots,d_Dc_1,qc_1/a_1,\ldots,qc_1/a_A;q)_{k}}{(q,b_1c_1,\ldots,b_Bc_1,qc_1/c_2,\ldots,qc_1/c_A;q)_{k}}
\bigg(\frac{a_1\cdots a_A}{c_1\cdots c_A}\bigg)^k
\\[1mm]
&\quad+\:idem(c_1;c_2,\ldots,c_A),
\end{align*}
where the integration is over a positively oriented unit circle such
that the origin and poles of $1/(d_1/z,\ldots,d_D/z;q)_{\infty}$ lie
inside the contour and the poles of
$1/(c_1z,\ldots,c_Az;q)_{\infty}$ lie outside the contour. Here the
notation $idem(c_1;c_2,\ldots,c_A)$ after an expression means the
sum of the $A-1$ expressions got from the front expression by
interchanging $c_1$ with $c_k$, where $k=2,3,\ldots,A$.
\end{lem}

Now we are ready to prove Theorem \ref{thm-a}.

\begin{proof}
Heine's transformation formulas of $_2\phi_1$ series (cf. \cite
[Appendix III.2 and III.3]{Gasper}) read as
\begin{align}
_{2}\phi_{1}\left[\begin{array}{c}
a,b\\
c
\end{array};q,\, z
\right] &=\frac{(c/a,az;q)_{\infty}}{(c,z;q)_{\infty}}
\,{_{2}\phi_{1}}\left[\begin{array}{c}
abz/c,a\\
az
\end{array};q,\, \frac{c}{a}
\right]
\label{Heine-a}\\[1mm]
&=\frac{(abz/c;q)_{\infty}}{(z;q)_{\infty}}
\,{_{2}\phi_{1}}\left[\begin{array}{c}
c/a,c/b\\
c
\end{array};q,\, \frac{abz}{c}
\right], \label{Heine-b}
\end{align}
where the basic hypergeometric series $_{r+1}\phi_{r}$ has been
defined by
$$
_{r+1}\phi_{r}\left[\begin{array}{c}
a_1,a_2,\ldots,a_{r+1}\\
b_1,b_2,\ldots,b_{r}
\end{array};q,\, z
\right] =\sum_{k=0}^{\infty}\frac{(a_1,a_2,\ldots, a_{r+1};q)_k}
{(q,b_1,b_2,\ldots,b_{r};q)_k}z^k.
$$ By means of \eqref{Heine-b}, we have the following two
relations:
\begin{align}
&{_{2}\phi_{1}}\left[\begin{array}{c}
a,aq/c\\
aq/b
\end{array};q,\, \frac{cq}{abz}
\right] =\frac{(q/z;q)_{\infty}}{(cq/abz;q)_{\infty}}
\,{_{2}\phi_{1}}\left[\begin{array}{c}
q/b,c/b\\
aq/b
\end{array};q,\, \frac{q}{z}
\right],
\label{Heine-c}\\[1mm]
&{_{2}\phi_{1}}\left[\begin{array}{c}
b,bq/c\\
bq/a
\end{array};q,\, \frac{cq}{abz}
\right] =\frac{(q/z;q)_{\infty}}{(cq/abz;q)_{\infty}}
\,{_{2}\phi_{1}}\left[\begin{array}{c}
q/a,c/a\\
bq/a
\end{array};q,\, \frac{q}{z}
\right]. \label{Heine-d}
\end{align}

Substituting \eqref{Heine-a}, \eqref{Heine-c}, and \eqref{Heine-d}
in turn into the three-term transformation formula of $_2\phi_1$
series (cf. \cite [Appendix III.32]{Gasper}):
\begin{align*}
_{2}\phi_{1}\left[\begin{array}{c}
a,b\\
c
\end{array};q,\, z
\right]
&=\frac{(b,c/a,az,q/az;q)_{\infty}}{(c,b/a,z,q/z;q)_{\infty}}
\,{_{2}\phi_{1}}\left[\begin{array}{c}
a,aq/c\\
aq/b
\end{array};q,\, \frac{cq}{abz}
\right]
\\[1mm]
&\quad+\frac{(a,c/b,bz,q/bz;q)_{\infty}}{(c,a/b,z,q/z;q)_{\infty}}
\,{_{2}\phi_{1}}\left[\begin{array}{c}
b,bq/c\\
bq/a
\end{array};q,\, \frac{cq}{abz}
\right],
\end{align*}
it is routine to show that
\begin{align*}
&(c/a,az;q)_{\infty}\,{_{2}\phi_{1}}\left[\begin{array}{c}
abz/c,a\\
az
\end{array};q,\, \frac{c}{a}
\right]
\notag\\[1mm]
&\:=\frac{(a,c/b,bz,q/bz;q)_{\infty}}{(a/b,cq/abz;q)_{\infty}}
\,{_{2}\phi_{1}}\left[\begin{array}{c}
q/a,c/a\\
bq/a
\end{array};q,\, \frac{q}{z}
\right]
\notag\\[1mm]
&\quad+\frac{(b,c/a,az,q/az;q)_{\infty}}{(b/a,cq/abz;q)_{\infty}}
\,{_{2}\phi_{1}}\left[\begin{array}{c}
q/b,c/b\\
aq/b
\end{array};q,\, \frac{q}{z}
\right].
\end{align*}
Take $(a,b,c,z)\to(-x/y,x/y,0,-y^2/x)$  in the last three-term
relation to obtain
\begin{align}
&(y;q)_{\infty}
\sum_{k=0}^{\infty}\frac{(-x/y;q)_k}{(q;q)_k(y;q)_k}q^{\binom{k}{2}}y^k
\notag\\[1mm]
&\:=\frac{(y,x/y,q/y;q)_{\infty}}{(-1;q)_{\infty}}
 \sum_{k=0}^{\infty}\frac{(qy/x;q)_k}{(q^2;q^2)_k}\bigg(-\frac{qx}{y^2}\bigg)^k
\notag\\[1mm]
&\quad+\frac{(-y,-x/y,-q/y;q)_{\infty}}{(-1;q)_{\infty}}
 \sum_{k=0}^{\infty}\frac{(-qy/x;q)_k}{(q^2;q^2)_k}\bigg(-\frac{qx}{y^2}\bigg)^k
 \label{eq:wei-aa}.
\end{align}
Choose $(A,B,D)=(2,1,0)$ and $(a_1,a_2,b_1,c_1,c_2)=(x,q,1,y,-y)$
 in Lemma \ref{lem-a} to get
\begin{align}
&\oint
\frac{(xz,q,qz,1/z;q)_{\infty}}{(y^2z^2;q^2)_{\infty}}\frac{dz}{2\pi
iz} \notag
\end{align}
\begin{align}
 &\:=\frac{(y,x/y,q/y;q)_{\infty}}{(-1;q)_{\infty}}
 \sum_{k=0}^{\infty}\frac{(qy/x;q)_k}{(q^2;q^2)_k}\bigg(-\frac{qx}{y^2}\bigg)^k
\notag\\[1mm]
&\quad+\frac{(-y,-x/y,-q/y;q)_{\infty}}{(-1;q)_{\infty}}
 \sum_{k=0}^{\infty}\frac{(-qy/x;q)_k}{(q^2;q^2)_k}\bigg(-\frac{qx}{y^2}\bigg)^k
 \label{eq:wei-bb}.
\end{align}
The combination of \eqref{eq:wei-aa} and \eqref{eq:wei-bb} engenders
\begin{align}
\oint
\frac{(xz,q,qz,1/z;q)_{\infty}}{(y^2z^2;q^2)_{\infty}}\frac{dz}{2\pi
iz}=(y;q)_{\infty}\sum_{k=0}^{\infty}\frac{(-x/y;q)_k}{(q;q)_k(y;q)_k}q^{\binom{k}{2}}y^k.
 \label{eq:wei-cc}
\end{align}

Recall Euler's another $q$-exponential formula (cf. \cite[Appendix
(II.1)]{Gasper}) and Jacobi's triple product  identity (cf.
\cite[Appendix (II.28)]{Gasper}) :
\begin{align}
&\:\:\sum_{k=0}^{\infty}\frac{z^k}{(q;q)_k}=\frac{1}{(z;q)_{\infty}},
\label{Euler-b}
\\[1mm]
&\sum_{k=-\infty}^{\infty}q^{\binom{k}{2}}z^{k}=(q,-z,-q/z;q)_{\infty}.
\label{Jacobi}
\end{align}
Employing \eqref{Euler-a}, \eqref{Euler-b}, and \eqref{Jacobi}, it
is not difficult to understand the following calculation:
\begin{align}
\oint
&\frac{(xz,q,qz,1/z;q)_{\infty}}{(y^2z^2;q^2)_{\infty}}\frac{dz}{2\pi
iz}
\notag\\[1mm]
&=\oint\sum_{j=0}^{\infty}\frac{q^{\binom{j}{2}}(-xz)^j}{(q;q)_j}\sum_{k=0}^{\infty}\frac{(yz)^{2k}}{(q^2;q^2)_k}
\sum_{\ell=-\infty}^{\infty}q^{\binom{\ell}{2}}(-1/z)^{\ell}\frac{dz}{2\pi
iz}
\notag\\[1mm]
&=\sum_{j,\,k\geq0}\frac{q^{j^2+2jk+2k^2-j-k}}{(q;q)_j(q^2;q^2)_k}x^{j}y^{2k}.
\label{eq:wei-dd}
\end{align}
By virtue of \eqref{eq:wei-cc} and \eqref{eq:wei-dd}, we are able to
arrive at \eqref{eq:wei-a} to complete the proof.
\end{proof}

\section{ Proof of
Theorems \ref{thm-d}-\ref{thm-f}}

In order to prove Theorems \ref{thm-d}-\ref{thm-f}, we draw support
on the following lemma (cf. \cite[Proposition 3.2]{Rosengren}).

\begin{lem}\label{lem-b}
Let $a,b,c,t$ be complex numbers such that $|t|<1$. Then
\begin{align*}
_{2}\phi_{1}\left[\begin{array}{c}
a,b\\
c
\end{array};q,\, t
\right]=\frac{(q;q)_{\infty}}{(c,t;q)_{\infty}}\oint
\frac{(abz,cz,qz/t,t/z;q)_{\infty}}{(az,bz,cz/t;q)_{\infty}}\frac{dz}{2\pi
iz},
\end{align*}
where the integral is over a positively oriented contour separating
the origin from all poles of the integrand.
\end{lem}

Firstly, we start to prove Theorem \ref{thm-d}.

\begin{proof}
Making the replacement $(a,q)\to(a^2,q^2)$ in the $q$-binomial
theorem:
\begin{align}
&{_{1}\phi_{0}}\left[\begin{array}{c}
a\\
-
\end{array};q,\, t
\right] =\frac{(at;q)_{\infty}}{(t;q)_{\infty}},
  \label{Bimo}
 \end{align}
it is ordinary to see that
\begin{align*}
&{_{2}\phi_{1}}\left[\begin{array}{c}
a,-a\\
-q
\end{array};q,\, t
\right] =\frac{(a^2t;q^2)_{\infty}}{(t;q^2)_{\infty}}.
 \end{align*}
Fix $(a,b,c,t)=(y^{\frac{1}{2}},-y^{\frac{1}{2}},-q,x)$ in Lemma
\ref{lem-b} and use the above identity to find
\begin{align}
\oint
\frac{(-yz,-qz,q,qz/x,x/z;q)_{\infty}}{(yz^2;q^2)_{\infty}(-qz/x;q)_{\infty}}\frac{dz}{2\pi
iz}=(-q;q)_{\infty}(qx,yx;q^2)_{\infty}.
 \label{eq:wei-a3}
\end{align}
On the base of \eqref{Euler-a}, \eqref{Euler-b}, \eqref{Jacobi}, and
\eqref{Bimo}, we have the following evaluation:
\begin{align}
&\oint
\frac{(-yz,-qz,q,qz/x,x/z;q)_{\infty}}{(yz^2;q^2)_{\infty}(-qz/x;q)_{\infty}}\frac{dz}{2\pi
iz}\notag\\[1mm]
&\quad=\oint\sum_{j=0}^{\infty}\frac{(x;q)_j}{(q;q)_j}\Big(-\frac{qz}{x}\Big)^j\sum_{k=0}^{\infty}\frac{q^{\binom{k}{2}}(yz)^{k}}{(q;q)_k}
\notag
\\[1mm]
&\qquad\times\sum_{\ell=0}^{\infty}\frac{(yz^2)^{\ell}}{(q^2;q^2)_{\ell}}\sum_{m=-\infty}^{\infty}(-1)^{m}q^{\binom{m}{2}}(x/z)^{m}\frac{dz}{2\pi
iz}
\notag\\[1mm]
&\quad=\sum_{j,\,k,\,\ell\,\geq0}\frac{(x;q)_j(-x)^{k+2\ell}y^{k+\ell}}{(q;q)_j(q;q)_k(q^2;q^2)_{\ell}}
q^{j+\binom{k}{2}+\binom{j+k+2\ell}{2}}. \label{eq:wei-b3}
\end{align}
Then the combination of \eqref{eq:wei-a3} with \eqref{eq:wei-b3}
produces \eqref{eq:wei-d} in the end.
\end{proof}

Secondly, we begins to prove Theorem \ref{thm-e}.

\begin{proof}
Replace $t$ by $-t$ in the $q$-binomial theorem \eqref{Bimo} to
obtain
\begin{align}
&{_{1}\phi_{0}}\left[\begin{array}{c}
a\\
-
\end{array};q,\, -t
\right] =\frac{(-at;q)_{\infty}}{(-t;q)_{\infty}}.
  \label{Bimo-a}
 \end{align}
The sum of \eqref{Bimo} and \eqref{Bimo-a} yields
\begin{align*}
&{_{2}\phi_{1}}\left[\begin{array}{c}
a,aq\\
q
\end{array};q^2,\, t^2
\right]=\frac{1}{2}\frac{(at;q)_{\infty}}{(t;q)_{\infty}}+\frac{1}{2}\frac{(-at;q)_{\infty}}{(-t;q)_{\infty}}.
 \end{align*}
Take $(q,a,b,c,t)\to(q^2,x,qx,q,y^2)$ in Lemma \ref{lem-b} and
utilize the last equation to get
\begin{align}
\oint
\frac{(qx^2z,qz,q^2,q^2z/y^2,y^2/z;q^2)_{\infty}}{(xz;q)_{\infty}(qz/y^2;q^2)_{\infty}}\frac{dz}{2\pi
iz}=\frac{(q;q^2)_{\infty}}{2}\big\{(xy,-y;q)_{\infty}+(-xy,y;q)_{\infty}\big\}.
 \label{eq:wei-c3}
\end{align}

Through \eqref{Euler-a}, \eqref{Euler-b}, \eqref{Jacobi}, and
\eqref{Bimo}, we are led to the following calculation:
\begin{align}
&\oint
\frac{(qx^2z,qz,q^2,q^2z/y^2,y^2/z;q^2)_{\infty}}{(xz;q)_{\infty}(qz/y^2;q^2)_{\infty}}\frac{dz}{2\pi
iz}\notag\\[1mm]
&\quad=\oint\sum_{j=0}^{\infty}\frac{(xz)^j}{(q;q)_j}\sum_{k=0}^{\infty}\frac{(x^2y^2;q^2)_k}{(q^2;q^2)_k}\Big(\frac{qz}{y^2}\Big)^k
\notag\\[1mm]
&\qquad\times\sum_{\ell=0}^{\infty}\frac{(-1)^{\ell}q^{2\binom{\ell}{2}}(qz)^{\ell}}{(q^2;q^2)_{\ell}}\sum_{m=-\infty}^{\infty}(-1)^{m}q^{2\binom{m}{2}}(y^2/z)^{m}\frac{dz}{2\pi
iz}
\notag\\[1mm]
&\quad=\sum_{j,\,k,\,\ell\,\geq0}\frac{(x^2y^2;q^2)_kx^jy^{2j+2\ell}}{(q;q)_j(q^2;q^2)_k(q^2;q^2)_{\ell}}
(-1)^{j+k}q^{(j+k+\ell)(j+k+\ell-1)+\ell^2+k}. \label{eq:wei-d3}
\end{align}
Substituting \eqref{eq:wei-d3} into \eqref{eq:wei-c3}, we can catch
 hold of \eqref{eq:wei-e1} to complete the proof.
\end{proof}

Thirdly, we plan to prove Theorem \ref{thm-f}.

\begin{proof}
The difference of \eqref{Bimo} and \eqref{Bimo-a} can be expressed
as
\begin{align*}
&{_{2}\phi_{1}}\left[\begin{array}{c}
aq,aq^2\\
q^3
\end{array};q^2,\, t^2
\right]=\frac{1-q}{2(1-a)t}\bigg\{\frac{(at;q)_{\infty}}{(t;q)_{\infty}}-\frac{(-at;q)_{\infty}}{(-t;q)_{\infty}}\bigg\}.
 \end{align*}
Perform the replacement $(q,a,b,c,t)\to(q^2, q^2x, q^3x,q^3,
y^2/q^2)$ in Lemma \ref{lem-b} and employ the upper formula to
discover
\begin{align}
&\oint
\frac{(q^5x^2z,q^3z,q^2,q^4z/y^2,y^2/zq^2;q^2)_{\infty}}{(q^2xz;q)_{\infty}(q^5z/y^2;q^2)_{\infty}}\frac{dz}{2\pi
iz}
\notag\\[1mm]
&\quad=\frac{(q;q^2)_{\infty}}{2(y/q-xy)}\big\{(xy,-y/q;q)_{\infty}-(-xy,y/q;q)_{\infty}\big\}.
 \label{eq:wei-e3}
\end{align}

In terms of \eqref{Euler-a}, \eqref{Euler-b}, \eqref{Jacobi}, and
\eqref{Bimo}, there is the following evaluation:
\begin{align}
&\oint
\frac{(q^5x^2z,q^3z,q^2,q^4z/y^2,y^2/zq^2;q^2)_{\infty}}{(q^2xz;q)_{\infty}(q^5z/y^2;q^2)_{\infty}}\frac{dz}{2\pi
iz}\notag\\[1mm]
&\quad=\oint\sum_{j=0}^{\infty}\frac{(q^2xz)^j}{(q;q)_j}\sum_{k=0}^{\infty}\frac{(x^2y^2;q^2)_k}{(q^2;q^2)_k}\Big(\frac{q^5z}{y^2}\Big)^k
\notag
\\[1mm]
&\qquad\times\sum_{\ell=0}^{\infty}\frac{(-1)^{\ell}q^{2\binom{\ell}{2}}(q^3z)^{\ell}}{(q^2;q^2)_{\ell}}\sum_{m=-\infty}^{\infty}(-1)^{m}q^{2\binom{m}{2}}(y^2/zq^2)^{m}\frac{dz}{2\pi
iz}
\notag\\[1mm]
&\quad=\sum_{j,\,k,\,\ell\,\geq0}\frac{(x^2y^2;q^2)_kx^jy^{2j+2\ell}}{(q;q)_j(q^2;q^2)_k(q^2;q^2)_{\ell}}
(-1)^{j+k}q^{(j+k+\ell)(j+k+\ell-1)+\ell^2+3k}. \label{eq:wei-f3}
\end{align}
Substituting \eqref{eq:wei-f3} into \eqref{eq:wei-e3}, we can grasp
hold of \eqref{eq:wei-f} in the end.
\end{proof}

\section{The bisection method and
multidimensional Rogers--Ramanujan type identities}

In this section, we shall use the bisection method to establish
several new multidimensional Rogers--Ramanujan type identities.

\begin{thm}
\begin{align}
&\sum_{j,\,k\geq0}\frac{q^{4j^2+4jk+2k^2-j}}{(q;q)_{2j}(q^2;q^2)_k}
=\frac{(q^8,-q^3,-q^5;q^8)_{\infty}}{(q^2;q^2)_{\infty}},
 \label{eq:wei-a4}
\\[1mm]
&\sum_{j,\,k\geq0}\frac{q^{4j^2+4jk+2k^2+3j+2k}}{(q;q)_{1+2j}(q^2;q^2)_k}
=\frac{(q^8,-q,-q^7;q^8)_{\infty}}{(q^2;q^2)_{\infty}}.
 \label{eq:wei-b4}
\end{align}
\end{thm}

\begin{proof}
Replace $x$ by $-x$ in  Cao and Wang's formula \eqref{wang-b} to
obtain
\begin{align}
\sum_{j,\,k\geq0}\frac{q^{j^2+2jk+2k^2+k}}{(q;q)_j(q^2;q^2)_k}(-1)^jx^{j+2k}
=(qx;q)_{\infty}.
 \label{eq:wei-c4}
\end{align}
The sum of \eqref{wang-b} and  \eqref{eq:wei-c4} engenders
\begin{align}
\sum_{j,\,k\geq0}\frac{q^{4j^2+4jk+2k^2+k}}{(q;q)_{2j}(q^2;q^2)_k}x^{2j+2k}
=\frac{1}{2}\Big\{(-qx;q)_{\infty}+(qx;q)_{\infty}\Big\}.
\label{bisection-a}
\end{align}
Notice a known relation (cf. \cite[Equations (1.2a)]{CW}):
\begin{align}
(-q;q^2)_{\infty}+(q;q^2)_{\infty}
=\frac{2}{(q^4;q^4)_{\infty}}(q^{16},-q^{6},-q^{10};q^{16})_{\infty}
 \label{WCY-a}.
\end{align}
Combing the $x=q^{-\frac{1}{2}}$ case of \eqref{bisection-a} and
\eqref{WCY-a}, we derive the equation \eqref{eq:wei-a4}.

The difference of \eqref{wang-b} and  \eqref{eq:wei-c4} produces
\begin{align}
\sum_{j,\,k\geq0}\frac{q^{4j^2+4jk+2k^2+4j+3k+1}}{(q;q)_{1+2j}(q^2;q^2)_k}x^{1+2j+2k}
=\frac{1}{2}\Big\{(-qx;q)_{\infty}-(qx;q)_{\infty}\Big\}.
\label{bisection-b}
\end{align}
Notice another known relation (cf. \cite[Equations (1.2b)]{CW}):
\begin{align}
&(-q;q^2)_{\infty}-(q;q^2)_{\infty}
=\frac{2q}{(q^4;q^4)_{\infty}}(q^{16},-q^{2},-q^{14};q^{16})_{\infty}.
 \label{WCY-b}
\end{align}
Combing the $x=q^{-\frac{1}{2}}$ case of \eqref{bisection-b} with
\eqref{WCY-b}, we  deduce the identity \eqref{eq:wei-b4}.
\end{proof}

\begin{thm}Let $m$ be a nonnegative integer. Then
\begin{align}
&\sum_{j,\,k,\,\ell\,\geq0}\frac{(q;q^2)_k(-1)^{k}}{(q;q)_{2j}(q^2;q^2)_k(q^2;q^2)_{\ell}}
q^{(2j+k+\ell)(2j+k+\ell-1)+\ell^2+j+k-2m(j+\ell)}
\notag\\[1mm]
&\quad=\frac{(-q^{-m};q)_m(q^8,-q^3,-q^5;q^8)_{\infty}}{(q^2;q^2)_{\infty}},
\label{eq:wei-d4}\\[1mm]
&\sum_{j,\,k,\,\ell\,\geq0}\frac{(q;q^2)_k(-1)^{k}}{(q;q)_{1+2j}(q^2;q^2)_k(q^2;q^2)_{\ell}}
q^{(1+2j+k+\ell)(2j+k+\ell)+\ell^2+j+k-m(1+2j+2\ell)}
\notag\\[1mm]
&\quad=\frac{(-q^{-m};q)_m(q^8,-q,-q^7;q^8)_{\infty}}{(q^2;q^2)_{\infty}}.
 \label{eq:wei-e4}
\end{align}
\end{thm}

\begin{proof}
Replace $x$ by $-x$ in  \eqref{eq:wei-g} to gain
\begin{align}
&\sum_{j,\,k,\,\ell\,\geq0}\frac{(x^2;q^2)_k\,(-x)^j(-1)^{k}}{(q;q)_j(q^2;q^2)_k(q^2;q^2)_{\ell}}
q^{(j+k+\ell)(j+k+\ell-1)+\ell^2+k-m(j+2\ell)}
\notag\\
&\quad=(-q^{-m};q)_m(x;q)_{\infty}.
 \label{eq:wei-f4}
\end{align}
According to \eqref{WCY-a} and the $x=q^{\frac{1}{2}}$ case of the
sum of \eqref{eq:wei-g} and  \eqref{eq:wei-f4}, we can certify
\eqref{eq:wei-d4}. In accordance with \eqref{WCY-b} and the
$x=q^{\frac{1}{2}}$ case of the difference of \eqref{eq:wei-g} and
\eqref{eq:wei-f4}, we may verify \eqref{eq:wei-e4}.
\end{proof}

\begin{thm}
\begin{align}
&\sum_{j,\,k,\,\ell\,\geq0}\frac{(q^{-1};q^2)_k(-1)^{k}}{(q;q)_{2j}(q^2;q^2)_k(q^2;q^2)_{\ell}}
q^{(2j+k+\ell)(2j+k+\ell-1)+\ell^2+j+3k+2\ell}
\notag\\[1mm]
&\quad=\frac{(q^8,-q^3,-q^5;q^8)_{\infty}}{(q^2;q^2)_{\infty}},
\label{eq:wei-g4}\\[1mm]
&\sum_{j,\,k,\,\ell\,\geq0}\frac{(q^{-1};q^2)_k(-1)^{k}}{(q;q)_{1+2j}(q^2;q^2)_k(q^2;q^2)_{\ell}}
q^{(1+2j+k+\ell)(2j+k+\ell)+\ell^2+j+3k+2\ell}
\notag\\[1mm]
&\quad=\frac{(q^8,-q,-q^7;q^8)_{\infty}}{(q^2;q^2)_{\infty}}.
 \label{eq:wei-h4}
\end{align}
\end{thm}

\begin{proof}
The $m=0$ case of \eqref{eq:wei-h} can be written as
\begin{align}
\sum_{j,\,k,\,\ell\,\geq0}\frac{(x^2;q^2)_k\,x^j(-1)^{k}}{(q;q)_j(q^2;q^2)_k(q^2;q^2)_{\ell}}
q^{(j+k+\ell)(j+k+\ell-1)+\ell^2+j+3k+2\ell}=(-qx;q)_{\infty}.
 \label{eq:wei-i4}
\end{align}
Replacing $x$ by $-x$ in \eqref{eq:wei-i4}, we have
\begin{align}
\sum_{j,\,k,\,\ell\,\geq0}\frac{(x^2;q^2)_k\,(-x)^j(-1)^{k}}{(q;q)_j(q^2;q^2)_k(q^2;q^2)_{\ell}}
q^{(j+k+\ell)(j+k+\ell-1)+\ell^2+j+3k+2\ell}=(qx;q)_{\infty}.
 \label{eq:wei-j4}
\end{align}
Via \eqref{WCY-a} and the $x=q^{-\frac{1}{2}}$ case of the sum of
\eqref{eq:wei-i4} and  \eqref{eq:wei-j4}, we find the equation
\eqref{eq:wei-g4}. Through \eqref{WCY-b} and the
$x=q^{-\frac{1}{2}}$ case of the difference of \eqref{eq:wei-i4} and
\eqref{eq:wei-j4}, we discover the identity \eqref{eq:wei-h4}.
\end{proof}

\begin{thm} Let $x$ be a complex number. Then
\begin{align}
&\sum_{j,\,k,\,\ell\,\geq0}\frac{(x;q^2)_k(-1)^{k}x^{-\ell}}{(q;q)_j(q^2;q^2)_k(q^2;q^2)_{\ell}}
q^{(j+k+\ell)(j+k+\ell-1)+\ell^2+j+k+2\ell}
\notag\\[1mm]
&\quad=\frac{(-qx,-q^{3}/x;q^4)_{\infty}}{(q^2;q^4)_{\infty}},
\label{eq:wei-k4}
\\[1mm]
&\sum_{j,\,k,\,\ell\,\geq0}\frac{(x;q^2)_k(-1)^{k}x^{-\ell}}{(q;q)_j(q^2;q^2)_k(q^2;q^2)_{\ell}}
q^{(j+k+\ell)(j+k+\ell-1)+\ell^2+2j+3k+4\ell}
\notag\\[1mm]
&\quad =\frac{(-q^{3}x,-q^{5}/x;q^4)_{\infty}}{(q^2;q^4)_{\infty}}.
\label{eq:wei-l4}
\end{align}
\end{thm}

\begin{proof}
Making the replacement $(x,y)\to(-x/q,q/x^{\frac{1}{2}})$ in
\eqref{eq:wei-e1} and using Jacobi's triple product identity
\eqref{Jacobi}, there holds the following calculation:
\begin{align*}
&\sum_{j,\,k,\,\ell\,\geq0}\frac{(x;q^2)_k(-1)^{k}x^{-\ell}}{(q;q)_j(q^2;q^2)_k(q^2;q^2)_{\ell}}
q^{(j+k+\ell)(j+k+\ell-1)+\ell^2+j+k+2\ell}
\notag\\
&\quad=\frac{(q;q^2)_{\infty}}{2}\Big\{(-x^{\frac{1}{2}},-q/x^{\frac{1}{2}};q)_{\infty}+(x^{\frac{1}{2}},q/x^{\frac{1}{2}};q)_{\infty}\Big\}
\\
&\quad=\frac{(q;q^2)_{\infty}}{2(q;q)_{\infty}}\Big\{\sum_{k=-\infty}^{\infty}q^{\binom{k}{2}}x^{\frac{k}{2}}+\sum_{k=-\infty}^{\infty}(-1)^kq^{\binom{k}{2}}x^{\frac{k}{2}}\Big\}
\notag\\
&\quad=\frac{1}{(q^2;q^2)_{\infty}}\sum_{k=-\infty}^{\infty}q^{\binom{2k}{2}}x^{k}
\\
&\quad =\frac{(-qx,-q^{3}/x;q^4)_{\infty}}{(q^2;q^4)_{\infty}}.
\end{align*}
So we complete the proof of \eqref{eq:wei-k4} .

Performing the replacement $(x,y)\to(-x/q^2,q^2/x^{\frac{1}{2}})$ in
\eqref{eq:wei-f} and utilizing Jacobi's triple product identity
\eqref{Jacobi}, there is  the following evaluation:
\begin{align*}
&\sum_{j,\,k,\,\ell\,\geq0}\frac{(x;q^2)_k(-1)^{k}x^{-\ell}}{(q;q)_j(q^2;q^2)_k(q^2;q^2)_{\ell}}
q^{(j+k+\ell)(j+k+\ell-1)+\ell^2+2j+3k+4\ell}
\notag\\
&\quad=\frac{(q;q^2)_{\infty}}{2(q/x^{\frac{1}{2}}+x^{\frac{1}{2}})}\Big\{(-x^{\frac{1}{2}},-q/x^{\frac{1}{2}};q)_{\infty}-(x^{\frac{1}{2}},q/x^{\frac{1}{2}};q)_{\infty}\Big\}
\notag
\\
&\quad=\frac{(q;q^2)_{\infty}}{2(q/x^{\frac{1}{2}}+x^{\frac{1}{2}})(q;q)_{\infty}}\Big\{\sum_{k=-\infty}^{\infty}q^{\binom{k}{2}}x^{\frac{k}{2}}-\sum_{k=-\infty}^{\infty}(-1)^kq^{\binom{k}{2}}x^{\frac{k}{2}}\Big\}
\notag
\end{align*}
\begin{align*}
&\quad=\frac{1}{(q/x^{\frac{1}{2}}+x^{\frac{1}{2}})(q^2;q^2)_{\infty}}\sum_{k=-\infty}^{\infty}q^{\binom{1+2k}{2}}x^{\frac{1+2k}{2}}
\\
&\quad =\frac{(-q^{3}x,-q^{5}/x;q^4)_{\infty}}{(q^2;q^4)_{\infty}}.
\end{align*}
Thus we complete the proof of \eqref{eq:wei-l4}.
\end{proof}

\section{Declarations}
{\bf{Conflicts of interest:}} No potential conflict of interest was
reported by the authors.\\
{\bf{ Availability of data and material:}} Not
applicable. \\
{\bf{Code availability:}} Not applicable.\\
{\bf{Funding:}} The work is supported by the National Natural
Science Foundation of China (No. 12571349).


\end{document}